\documentclass[11pt]{article}
\usepackage{amsmath,amsthm,amssymb,amsfonts}
\usepackage[numbers,sort&compress]{natbib}
\usepackage{color,colordvi}
\usepackage{fullpage}

\newtheorem{theorem}{Theorem}[section]

\newtheorem{lemma}[theorem]{Lemma}
\newtheorem{corollary}[theorem]{Corollary}

\title{Spectral expansion properties of pseudorandom bipartite graphs}
\author{Dandan Fan$^{a,b}$, Xiaofeng Gu$^c$, Huiqiu Lin$^{a}$
\\[2mm]
\small\it $^a$ School of Mathematics, East China University of Science and Technology, \\ \small\it   Shanghai 200237, China\\[1mm]
\small\it $^b$ College of Mathematics and Physics, Xinjiang Agricultural University,\\ \small\it Urumqi, Xinjiang 830052, China\\[1mm]
\small\it $^c$ Department of Computing and Mathematics, University of West Georgia, \\ \small\it Carrollton, GA 30118, USA
\\ \footnotesize {\it Email}: {\tt ddfan0526@163.com} (D. Fan), {\tt xgu@westga.edu} (X. Gu), {\tt huiqiulin@126.com} (H. Lin)
}

\begin{document}
\date{}
\maketitle
\noindent

\begin{abstract}
An $(a,b)$-biregular bipartite graph is a bipartite graph with bipartition $(X, Y)$ such that each vertex in $X$ has degree $a$ and each vertex in $Y$ has degree $b$. By the bipartite expander mixing lemma, biregular bipartite graphs have nice pseudorandom and expansion properties when the second largest adjacency eigenvalue is not large. In this paper, we prove several explicit properties of biregular bipartite graphs from spectral perspectives. In particular, we show that for any $(a,b)$-biregular bipartite graph $G$, if the spectral gap is greater than $\frac{2(k-1)}{\sqrt{(a+1)(b+1)}}$, then $G$ is $k$-edge-connected; and if the spectral gap is at least $\frac{2k}{\sqrt{(a+1)(b+1)}}$, then $G$ has at least $k$ edge-disjoint spanning trees. We also prove that if the spectral gap is at least $\frac{(k-1)\max\{a,b\}}{2\sqrt{ab - (k-1)\max\{a,b\}}}$, then $G$ is $k$-connected for $k\ge 2$; and if the spectral gap is at least $\frac{6k+2\max\{a,b\}}{\sqrt{(a-1)(b-1)}}$, then $G$ has at least $k$ edge-disjoint spanning 2-connected subgraphs. We have stronger results in the paper.
\end{abstract}



\section{Introduction}

An {\it expander} graph is a sparse graph that has strong connectivity properties, quantified using vertex, edge or spectral expansion. Expander graphs have many applications in computer science such as complexity theory, communication network, and the theory of error-correcting codes. Expanders are usually regular graphs, however, unbalanced bipartite expanders are better suited for applications in many situations. In particular, an idea of bipartite expander graphs called {\it magical} graphs can be used to construct super concentrators and good error correcting codes, see~\cite{HLW06} for more details. In this paper, we study bipartite expander properties from spectral perspectives.

Let $\lambda_i:=\lambda_i(G)$ denote the $i$th largest eigenvalue of the adjacency matrix of a simple graph $G$ on $n$ vertices, for $i=1,2,\cdots, n$. By the Perron-Frobenius Theorem, $\lambda_1$ is always positive (unless $G$ has no edges) and $|\lambda_i|\le \lambda_1$ for all $i\ge 2$. If $G$ is not a complete graph, then $\lambda_2$ is nonnegative. Let $\lambda = \max_{2\le i\le n} |\lambda_i| =\max\{|\lambda_2|, |\lambda_n|\}$, that is, $\lambda$ is the second largest absolute eigenvalue. For a $d$-regular graph, it is well known that $\lambda_1=d$. A $d$-regular graph on $n$ vertices with the second largest absolute eigenvalue $\lambda$ is called an {\it $(n,d,\lambda)$-graph}. The Alon-Boppana bound~\cite{Alon86} states that $\lambda \ge 2\sqrt{d-1} - o(1)$. For non-bipartite $d$-regular graph $G$, if $\lambda(G) \le 2\sqrt{d-1}$, then $G$ is a {\it Ramanujan} graph. Basically, Ramanujan graphs are those graphs attaining the Alon-Boppana bound, and thus are expander graphs with optimal asymptotic spectral properties.

A pseudorandom graph with $n$ vertices of edge density $p$ is a graph that behaves like a truly random graph $G(n, p)$. The first quantitative definition of pseudorandom graphs was introduced by Thomason~\cite{Thom85, Thom87} who defined jumbled graphs.  It is well known that an $(n, d, \lambda)$-graph  for which $\lambda = \Theta(\sqrt{d})$ is a very good pseudorandom graph, according to the celebrated expander mixing lemma. For more about pseudorandom graphs, see the survey~\cite{KrSu06} by Krivelevich and Sudakov.

\begin{theorem}[Expander mixing lemma]
Let $G$ be an $(n, d, \lambda)$-graph. Then for every two subsets $A$ and $B$ of $V(G)$, 
\begin{equation*}
\left| e(A, B) - \frac{d}{n} |A| |B| \right| \le \lambda \sqrt{|A| |B| \left(1-\frac{|A|}{n}\right)\left(1-\frac{|B|}{n}\right)},
\end{equation*}
where $e(A, B)$ denotes the number of edges with one end in $A$ and the other end in $B$ $($edges with both ends in $A\cap B$ are counted twice$)$.
\end{theorem}
The expander mixing lemma is usually attributed to Alon and Chung~\cite{AlCh88}. However, this idea was used earlier with a different form in the PhD thesis~\cite{Haem79} of Haemers. In fact, Haemers provided the expander mixing lemma for bipartite graphs.

An $(X, Y)$-bipartite graph is a bipartite graph with bipartition $(X, Y)$. An $(X, Y)$-bipartite graph is called {\it $(a,b)$-biregular} if each vertex in $X$ has degree $a$ and each vertex in $Y$ has degree $b$. 
Note that we must have $a|X|=b|Y|$. It is known that $\lambda_1(G) =\sqrt{ab}$ and $\lambda_n(G) =-\sqrt{ab}$ for any $(a,b)$-biregular bipartite graph $G$ on $n$ vertices. Godsil and Mohar~\cite{GM88}, Feng and Li~\cite{FL96}, and Li and Sol\'e~\cite{LS96} proved an analog of Alon-Boppana bound that $\lambda_2 \ge \sqrt{a-1} +\sqrt{b-1} -o(1)$ for $(a,b)$-biregular bipartite graphs. An $(a,b)$-biregular bipartite graph $G$ is {\it Ramanujan} if $\lambda_2(G) \le \sqrt{a-1} +\sqrt{b-1}$. In other words, biregular bipartite Ramanujan graphs are those who attain the optimal lower bound of $\lambda_2$. When $a=b$, we obtain regular bipartite Ramanujan graphs. It is proven in \cite{MSS15} that there are infinite many $(a,b)$-biregular bipartite Ramanujan graphs for $a,b\ge 3$. Explicit examples of biregular bipartite Ramanujan graphs were constructed in~\cite{EFMP23}. For random biregular bipartite graphs, the second largest eigenvalue has been studied in \cite{BDH22} and \cite{Zhu23}, while the the spectral distribution was studied in \cite{DJ16} and \cite{Tran20}.

For convenience, an $(a,b)$-biregular $(X, Y)$-bipartite graph is also called an {\it $(X, Y, a, b)$-graph}. The following bipartite expander mixing lemma first appeared in~\cite[Theorem~3.1.1]{Haem79} and also in \cite[Theorem 5.1]{Haem95}, but this exact form and the proof can be found in~\cite{DeSV12}.

\begin{theorem}[Bipartite expander mixing lemma]
\label{bieml}
Let $G$ be an $(X, Y, a, b)$-graph on at least $3$ vertices. Then for every two subsets $A\subseteq X$ and $B\subseteq Y$, 
\begin{equation*}
\left| e(A, B) -\frac{\sqrt{ab}}{\sqrt{|X| |Y|}} |A| |B| \right| \le \lambda_2 \sqrt{|A| |B| \left(1-\frac{|A|}{|X|}\right)\left(1-\frac{|B|}{|Y|}\right)}.
\end{equation*}
\end{theorem}

Note that $a|X| =b|Y|$ and it follows that $\frac{\sqrt{ab}}{\sqrt{|X| |Y|}} =\frac{b}{|X|} =\frac{a}{|Y|}$. Since the spectrum of any bipartite graph is symmetric about zero, $\lambda_2$ actually is the third largest absolute eigenvalue. As mentioned in \cite{DeSV12}, the bipartite expander mixing lemma is especially applicable when $\lambda_2$ is small. In this case, for any two subsets $A\subseteq X$ and $B\subseteq Y$, $e(A, B)$ is close to $\frac{\sqrt{ab}}{\sqrt{|X| |Y|}} |A| |B|$. The edge distribution is similar to the random bipartite graph and thus $(X, Y, a, b)$-graphs are {\it pseudorandom bipartite graphs}.

\section{Main results}\label{sect:main}
In this section, we discover several explicit expander-related properties of biregular bipartite graphs, including edge connectivity, vertex connectivity, spanning tree packing number, as well as rigidity in the Euclidean plane. These extend results of pseudorandom graphs to pseudorandom bipartite graphs.

\subsection{Edge connectivity}
The \textit{edge connectivity} of a graph $G$, denoted by $\kappa'(G)$, is the minimum number of edges whose removal produces a disconnected graph. If $\kappa'(G)\ge k$, then $G$ is $k$-edge-connected. Krivelevich and Sudakov~\cite{KrSu06} showed that any $(n,d,\lambda)$-graph $G$ with $d-\lambda\ge 2$ is $d$-edge-connected. The result was improved by Cioab\u{a}~\cite{Cioa10} using $\lambda_2$. In fact, Cioab\u{a}~\cite{Cioa10} proved that for $d \ge k\ge 2$, if $G$ is a $d$-regular graph of order $n$ with $\lambda_2(G)< d - \frac{n(k-1)}{(d+1)(n-d-1)}$, then $\kappa'(G)\ge k$; and in particular, if $\lambda_2(G)< d-\frac{2(k-1)}{d+1}$, then $\kappa'(G)\ge k$. These imply corresponding results in pseudorandom graphs. We would also like to point out that the results were generalized to irregular cases in \cite{GLLY12, LiHL13}. 
Here we prove an analogue for biregular bipartite graphs.

\begin{theorem}\label{thm:edge-conn}
Let $G$ be an $(X, Y, a, b)$-graph with $a,b\ge k$. If 
\begin{eqnarray}
\label{kappa1}
\lambda_2 < \sqrt{ab} - \frac{(k-1)\sqrt{|X||Y|}}{2\sqrt{\Big\lceil\frac{a+1}{2}\Big\rceil\Big\lceil\frac{b+1}{2}\Big\rceil\Big(|X|-\Big\lceil\frac{b+1}{2}\Big\rceil\Big)\Big(|Y|-\Big\lceil\frac{a+1}{2}\Big\rceil\Big)}},
\end{eqnarray}
then $\kappa'(G)\ge k$. In particular, if
\begin{eqnarray}
\label{kappa2}
\lambda_2 < \sqrt{ab} - \frac{k-1}{\sqrt{\Big\lceil\frac{a+1}{2}\Big\rceil\Big\lceil\frac{b+1}{2}\Big\rceil}},
\end{eqnarray}
then $\kappa'(G)\ge k$.
\end{theorem}

Note that $\lambda_1 -\lambda_2$ is usually called the {\it spectral gap} and $\lambda_1 = \sqrt{ab}$ for  $(X, Y, a, b)$-graphs. The spectral gap of biregular bipartite graphs has been shown to be very useful in error correcting codes, matrix completion and community detection (see \cite{BDH22,BV20,SS96}, among others). The above theorem indicates that if the spectral gap is greater than $\frac{k-1}{\sqrt{\lceil\frac{a+1}{2}\rceil\lceil\frac{b+1}{2}\rceil}}$, then $G$ is $k$-edge-connected. 

\subsection{Vertex connectivity}
The \textit{vertex connectivity} of a graph $G$, denoted by $\kappa(G)$, is the minimum number of vertices whose removal produces a disconnected graph. If $\kappa(G)\ge k$, then $G$ is \textit{$k$-connected}. By a result of Fiedler~\cite{F73} on algebraic connectivity, for a non-complete graph $G$, $\kappa(G)\ge \delta -\lambda_2$, where $\delta$ is the minimum degree of $G$. For any $(n,d,\lambda)$-graph $G$ with $d\leq n/2$, Krivelevich and Sudakov\cite{KrSu06} showed that $\kappa(G)\ge d-36\lambda^2/d$, and the result was generalized in \cite{Gu21}. We study the vertex connectivity of biregular bipartite graphs.

\begin{theorem}\label{thm:conn}
Let $G$ be an $(X,Y,a,b)$-graph with $a,b\ge k\ge 2$. If 
\begin{eqnarray}\label{eqn:kcon}
\lambda_2\le \sqrt{ab} -\frac{(k-1)\max\{a,b\}}{2\sqrt{ab - (k-1)\max\{a,b\}}},
\end{eqnarray}
then $\kappa(G)\ge k$.
\end{theorem}

In general, the theorem is not comparable with Fiedler's result. But in many cases, it seems our theorem can be better. For $d$-regular bipartite graphs $G$, the theorem indicates that if $\lambda_2\le d -\frac{(k-1)\sqrt{d}}{2\sqrt{d -k+1}}$, then $\kappa(G)\ge k$, which is stronger than Fiedler's result when $\kappa(G)\le 3d/4$. For any $(a,b)$-biregular bipartite graph with $a\ge b\ge k\ge 2$, it implies that if $\lambda_2\le \sqrt{a}\left(\sqrt{b} -\frac{k-1}{2\sqrt{b -k+1}}\right)$, then $\kappa(G)\ge k$. Apparently, the result would be better if $G$ is unbalanced.

\subsection{Edge-disjoint spanning trees}
The \textit{spanning tree packing number} (or simply \textit{STP number}) of a graph $G$, denoted by $\tau(G)$, is the maximum number of edge-disjoint spanning trees contained in $G$. The well known spanning tree packing theorem (Theorem~\ref{Nash-Williams-Tutte} in the next section) by Nash-Williams~\cite{Nash-Williams} and independently by Tutte~\cite{Tutte} provides a structural characterization of graphs $G$ with $\tau(G)\geq k$. The spanning tree packing number has been well studied from spectral perspectives. In fact, by the connection between edge connectivity and the spanning tree packing number, the result of Cioab\u{a}~\cite{Cioa10} implies that for any $d$-regular connected graph $G$, if $\lambda_{2}(G)<d-\frac{2(2k-1)}{d+1}$ for $d\geq 2k\geq 4$, then $\kappa'(G)\ge 2k$, and consequently $\tau(G)\geq k$. Cioab\u{a} and Wong \cite{CiWo12} conjectured that the sufficient condition can be improved to $\lambda_{2}(G)<d-\frac{2k-1}{d+1}$, and they confirmed the conjecture for $k=2,3$. Gu et al.~\cite{GLLY12} extended it to general graphs and provided some partial results. The conjecture of Cioab\u{a} and Wong \cite{CiWo12} was completely settled in \cite{LHGL14}, in which the authors proved for general graphs $G$ that if $\lambda_2(G)< \delta-\frac{2k-1}{\delta+1}$ for $\delta\geq 2k\geq 4$, then $\tau(G)\geq k$, where $\delta$ denotes the minimum degree of $G$. Stronger results via eigenvalues of other matrices were also obtained in \cite{LHGL14}. Most recently, the results of \cite{LHGL14} have been shown to be essentially best possible in \cite{coppww22} by constructing extremal graphs. Later on, a similar result in pseudorandom graphs was proved in \cite{Gu21} via expander mixing lemma, and a spectral radius condition was obtained in \cite{Fan-Gu-Lin}.

We study the problem in biregular bipartite graphs. Indeed, the result in \cite{LHGL14} for general graphs can be applied to bipartite graphs, however, we can prove the following improved theorem.

\begin{theorem}\label{thm:stp}
Let $G$ be an $(X,Y,a,b)$-graph with $a,b\ge 2k$. If 
\begin{eqnarray*}
\lambda_2\le \sqrt{ab} -\frac{k}{\sqrt{\lceil\frac{a+1}{2}\rceil\lceil\frac{b+1}{2}\rceil}},
\end{eqnarray*}
then $\tau(G)\geq k$.
\end{theorem}
The theorem indicates that if the spectral gap is at least $\frac{k}{\sqrt{\lceil\frac{a+1}{2}\rceil\lceil\frac{b+1}{2}\rceil}}$, then $G$ has at least $k$ edge-disjoint spanning trees.

\subsection{Rigidity in the Euclidean plane}
A \textit{$d$-dimensional bar-and-joint framework} $(G,p)$ is the combination of an undirected simple graph $G=(V(G),E(G))$ and a map $p: V(G)\rightarrow \mathbb{R}^{d}$ that  assigns a point in $\mathbb{R}^{d}$ to each vertex of $G$. Let $\|\cdot\|$ denote the Euclidean norm in $\mathbb{R}^d$. Two frameworks $(G,p)$ and $(G,q)$ are said to be \textit{equivalent} (resp., \textit{congruent}) if $\|p(u)-p(v)\|=\|q(u)-q(v)\|$ holds for all $uv\in E(G)$ (resp., for all $u,v\in V(G)$).  A framework $(G,p)$ is \textit{generic} if the coordinates of its points are algebraically independent over $\mathbb{Q}$. The framework $(G,p)$ is \textit{rigid} in $\mathbb{R}^{d}$ if there exists an $\varepsilon>0$ such that every framework $(G,q)$ equivalent to $(G,p)$ satisfying $\|p(u)-q(u)\|<\varepsilon$ for all $u\in V(G)$ is actually congruent to $(G,p)$. A generic framework $(G,p)$ is rigid in $\mathbb{R}^d$ if and only if every generic framework of $G$ is rigid in $\mathbb{R}^d$. Hence, generic rigidity can be considered as a property of the underlying graph. We say that a graph $G$ is  \textit{rigid} in $\mathbb{R}^d$ if every/some generic framework of $G$ is rigid in $\mathbb{R}^d$, and it is \textit{redundantly rigid} in $\mathbb{R}^d$ if $G-e$ is rigid in $\mathbb{R}^d$ for every $e\in E(G)$.

A framework $(G, p)$ is \textit{globally rigid} in $\mathbb{R}^d$  if every framework that is equivalent to $(G, p)$ is congruent to $(G, p)$. It was proved in \cite{GortlerThurstonHealey10} that if there exists a generic framework $(G,p)$ in $\mathbb{R}^d$ that is globally rigid, then any other generic framework $(G,q)$ in $\mathbb{R}^d$ is also globally rigid. Thus we say that $G$ is {\it globally rigid} in $\mathbb{R}^d$ if there exists a globally rigid generic framework $(G,p)$ in $\mathbb{R}^d$. By a combination of the results in \cite{Conn05, JaJo05}, a graph $G$ is globally rigid in $\mathbb{R}^2$ if and only if $G$ is 3-connected and redundantly rigid, or $G$ is a complete graph on at most three vertices.

In this paper, we consider rigidity in the Euclidean plane $\mathbb{R}^2$ only. Rigidity in $\mathbb{R}^2$ has been well studied. For a subset $X\subseteq V(G)$, let $G[X]$ be the subgraph of $G$ induced by $X$ and $E(X)$ denote the edge set of $G[X]$. A graph $G$ is {\it sparse} if $|E(X)|\le 2|X|-3$ for every $X\subseteq V(G)$ with $|X|\ge 2$. If in addition $|E(G)|=2|V(G)|-3$, then $G$ is called a {\it Laman graph}. A graph $G$ is {\it rigid} in $\mathbb{R}^2$ if and only if $G$ contains a spanning Laman subgraph. This characterization was first discovered by Pollaczek-Geiringer \cite{Poll1927} and rediscovered by Laman~\cite{Lama70}, and thus is also called the {\it Geiringer-Laman condition}. Lov\'{a}sz and Yemini \cite{LoYe82} gave a new characterization of rigid graphs and showed that every $6$-connected graph is globally rigid in $\mathbb{R}^2$. This result was improved by \cite{JaJo09} using an idea of mixed connectivity and later by \cite{GMRWY21} using a relaxation of connectivity called {\it essential connectivity}.

Cioab\u{a}, Dewar and Gu~\cite{CDG21} studied spectral conditions for rigidity and global rigidity in $\mathbb{R}^2$ via algebraic connectivity. One of the results implies that for a graph $G$ with minimum degree $\delta \ge 6k$, if $\lambda_2(G) < \delta -2 -\frac{2k-1}{\delta -1}$, then $G$ contains at least $k$ edge-disjoint spanning rigid subgraphs. For biregular bipartite graphs, we prove the following different result.

\begin{theorem}\label{thm:rigid}
Let $G$ be an $(X,Y,a,b)$-graph with $a,b\ge 6k$. If 
\begin{eqnarray*}
\lambda_2\le \sqrt{ab} -\frac{3k+\max\{a,b\}}{\sqrt{\lceil\frac{a-1}{2}\rceil\lceil\frac{b-1}{2}\rceil}},
\end{eqnarray*}
then $G$ has at least $k$ edge-disjoint spanning rigid subgraphs.
\end{theorem}

It is well known that any rigid graph with at least 3 vertices is 2-connected. Thus Theorem~\ref{thm:rigid} implies a sufficient condition for a biregular bipartite graph containing $k$ edge-disjoint spanning 2-connected subgraphs, and so it can be considered as an extension of Theorem~\ref{thm:stp}.
\begin{corollary}
Let $G$ be an $(a,b)$-biregular bipartite graph with $a,b\ge 6k$. If 
$$\lambda_2\le \sqrt{ab} -\frac{3k+\max\{a,b\}}{\sqrt{\lceil\frac{a-1}{2}\rceil\lceil\frac{b-1}{2}\rceil}},$$
then $G$ has at least $k$ edge-disjoint spanning 2-connected subgraphs.
\end{corollary}

When $k=1$, we obtain a sufficient condition for rigid graphs.
\begin{corollary}
Let $G$ be an $(a,b)$-biregular bipartite graph with $a,b\ge 6$. If 
$$\lambda_2\le \sqrt{ab} -\frac{3+\max\{a,b\}}{\sqrt{\lceil\frac{a-1}{2}\rceil\lceil\frac{b-1}{2}\rceil}},$$
then $G$ is rigid.
\end{corollary}


We also discover a similar condition for global rigidity.
\begin{theorem}\label{thm:globally-rigid}
Let $G$ be an $(X,Y,a,b)$-graph with $a,b\ge 6$. If 
\begin{eqnarray*}
\lambda_2\leq\sqrt{ab} -\frac{3+\max\{a,b\}}{\sqrt{\lceil\frac{a-2}{2}\rceil\lceil\frac{b-2}{2}\rceil}},
\end{eqnarray*}
then $G$ is globally rigid.
\end{theorem}

The rigidity of regular Ramanujan graphs and regular bipartite Ramanujan graphs was studied in~\cite{CDGG23}. Recall that an $(a,b)$-biregular bipartite graph is Ramanujan if $\lambda_2\le \sqrt{a-1} +\sqrt{b-1}$. Our results imply the rigidity of $(a,b)$-biregular bipartite Ramanujan graphs for given values of $a,b$.

\section{The proofs of main results}
In this section, we will prove the theorems. One of the main tools is the bipartite expander mixing lemma. Let $V(G)=V$ in the proofs below.

\subsection{The proofs of Theorems~\ref{thm:edge-conn}, \ref{thm:conn} and \ref{thm:stp}}

\begin{proof}[\bf The proof of Theorem~\ref{thm:edge-conn}]
Suppose to the contrary that there is a nonempty proper subset $A$ of $V(G)$ such that $e(A, V-A)\le k-1$. First, notice that $A\cap X\neq \emptyset$, for otherwise, $A\subseteq Y$ and so $e(A, V-A)= b |A| \ge k$, a contradiction. Similarly $A\cap Y\neq \emptyset$. Let $A\cap X = S$, $A\cap Y = T$, $X-S =S'$ and $Y-T = T'$. Then $(V-A)\cap X = S'$ and $(V-A)\cap Y = T'$. It follows that $e(A, V-A) = e(S, T') + e(S', T)$.

We claim that $2|S|\ge b+1$. Otherwise, $2|S|\le b$, then $e(A, V-A) = a|S| + b|T|- 2e(S, T)\ge  a|S| + b|T| - 2|S||T| \ge   a|S|\ge k$, which also leads to a contradiction. Similarly, we have $2|S'|\ge b+1$, $2|T|\ge a+1$ and $2|T'|\ge a+1$. Let $|X| =x$, $|Y| =y$, $|S| =s$ and $|T| =t$. By Theorem~\ref{bieml},
\begin{equation*}
\begin{aligned}
e(S, T') 
&\ge \frac{\sqrt{ab}}{\sqrt{|X| |Y|}} |S| |T'| - \lambda_2 \sqrt{|S| |T'| \left(1-\frac{|S|}{|X|}\right)\left(1-\frac{|T'|}{|Y|}\right)}\\
&= \frac{\sqrt{ab}}{\sqrt{xy}}\cdot s(y-t) - \lambda_2 \sqrt{s(y-t) \left(1-\frac{s}{x}\right)\left(1-\frac{y-t}{y}\right)}
\\&= \frac{\sqrt{ab}}{\sqrt{xy}}\cdot s(y-t) - \frac{\lambda_2}{\sqrt{xy}} \sqrt{s(y-t)(x-s)t}.
\end{aligned}
\end{equation*}
Similarly,
\begin{equation*}
\begin{aligned}
e(S', T) 
&\ge \frac{\sqrt{ab}}{\sqrt{|X| |Y|}} |S'| |T| - \lambda_2 \sqrt{|S'| |T| \left(1-\frac{|S'|}{|X|}\right)\left(1-\frac{|T|}{|Y|}\right)}
\\&= \frac{\sqrt{ab}}{\sqrt{xy}}\cdot (x-s)t - \frac{\lambda_2}{\sqrt{xy}} \sqrt{(x-s)ts(y-t)}.
\end{aligned}
\end{equation*}
Thus
\begin{equation*}
\begin{aligned}
e(A, V-A) 
&= e(S, T') + e(S', T)
\\&\ge \frac{\sqrt{ab}}{\sqrt{xy}}\cdot \left( s(y-t) +  (x-s)t \right) - \frac{2\lambda_2}{\sqrt{xy}} \sqrt{s(x-s)t(y-t)}.
\\&\ge \frac{\sqrt{ab}}{\sqrt{xy}}\cdot 2\sqrt{s(x-s)t(y-t)} - \frac{2\lambda_2}{\sqrt{xy}} \sqrt{s(x-s)t(y-t)}.
\\&= \frac{\sqrt{ab}-\lambda_2}{\sqrt{xy}}\cdot 2\sqrt{s(x-s)t(y-t)}.
\end{aligned}
\end{equation*}
It is not hard to see that $s(x-s)$ is minimized at $s=\lceil\frac{b+1}{2}\rceil$ or $s=x-\lceil\frac{b+1}{2}\rceil$, and $t(y-t)$ is minimized at $t=\lceil\frac{a+1}{2}\rceil$ or $t=y-\lceil\frac{a+1}{2}\rceil$, and hence
\begin{equation*}
\begin{aligned}
e(A, V-A) 
&\ge \frac{2\sqrt{ab}-2\lambda_2}{\sqrt{xy}}\cdot \sqrt{\Big\lceil\frac{b+1}{2}\Big\rceil\Big(x-\Big\lceil\frac{b+1}{2}\Big\rceil\Big)\Big\lceil\frac{a+1}{2}\Big\rceil\Big(y-\Big\lceil\frac{a+1}{2}\Big\rceil\Big)}.
\end{aligned}
\end{equation*}
By (\ref{kappa1}), we have $e(A, V-A) > k-1$, a contradiction. Therefore $\kappa'(G)\ge k$.

To prove (\ref{kappa2}) is also sufficient, notice that $x = |S| + |S'| \ge 2\lceil\frac{b+1}{2}\rceil$ and so $1 -\frac{\lceil\frac{b+1}{2}\rceil}{x}\ge\frac{1}{2}$. 
Similarly, $1 -\frac{\lceil\frac{a+1}{2}\rceil}{y}\ge\frac{1}{2}$. Thus
\begin{equation*}
\begin{aligned}
e(A, V-A) 
&\ge \frac{2\sqrt{ab}-2\lambda_2}{\sqrt{xy}}\cdot \sqrt{\Big\lceil\frac{b+1}{2}\Big\rceil\Big(x-\Big\lceil\frac{b+1}{2}\Big\rceil\Big)\Big\lceil\frac{a+1}{2}\Big\rceil\Big(y-\Big\lceil\frac{a+1}{2}\Big\rceil\Big)}
\\&= 2(\sqrt{ab}-\lambda_2)\cdot \sqrt{\Big\lceil\frac{b+1}{2}\Big\rceil\Big(1 -\frac{\lceil\frac{b+1}{2}\rceil}{x}\Big)\Big\lceil\frac{a+1}{2}\Big\rceil\Big(1 -\frac{\lceil\frac{a+1}{2}\rceil}{y}\Big)}
\\&\ge (\sqrt{ab}-\lambda_2)\cdot \sqrt{\Big\lceil\frac{a+1}{2}\Big\rceil\Big\lceil\frac{b+1}{2}\Big\rceil}.
\end{aligned}
\end{equation*}
By (\ref{kappa2}), we have $e(A, V-A) > k-1$, a contradiction. Therefore $\kappa'(G)\ge k$.
\end{proof}

\begin{proof}[\bf The proof of Theorem~\ref{thm:conn}]
Suppose to the contrary that $\kappa(G)\le k-1$. Let $W$ be a vertex cut with $|W|\le k-1$ such that $|W\cap X| =k_1$, $|W\cap Y| =k_2$ and $k_1 +k_2 \le k-1$. Suppose that $A_1, A_2,\ldots, A_p$ are the components of $G-W$, where $p\geq 2$. Without loss of generality, we assume that $e(A_1,W)\leq e(A_2,W)\leq\cdots\leq e(A_p,W)$. On the one hand, we have $e(A_1,V-A_1)=e(A_1,W)\le \frac{k_1a +k_2b}{p}\leq \frac{k_1a +k_2b}{2}$.  On the other hand, let $W\cap X=W_1$, $W\cap Y=W_2$, $A_1\cap X=S$, $A_1\cap Y=T$, $X-(S\cup W_1)=S'$ and $Y-(T\cup W_2)=T'$. It is easy to verify that $|S|\geq b-k_1$, $|T|\geq a-k_2$, $|S'|\geq b-k_1$ and $|T'|\geq a-k_2$. Let $|X|=x$, $|Y|=y$, $|S|=s$ and $|T|=t$. Then, by Theorem~\ref{bieml}, we have
\begin{equation*}
\begin{aligned}
e(S, Y-T) 
&\ge \frac{\sqrt{ab}}{\sqrt{|X| |Y|}} |S| |Y-T| - \lambda_2 \sqrt{|S| |Y-T| \left(1-\frac{|S|}{|X|}\right)\left(1-\frac{|Y-T|}{|Y|}\right)}
\\&= \frac{\sqrt{ab}}{\sqrt{xy}}\cdot s(y-t) - \lambda_2 \sqrt{s(y-t) \left(1-\frac{s}{x}\right)\left(1-\frac{y-t}{y}\right)}
\\&= \frac{\sqrt{ab}}{\sqrt{xy}}\cdot s(y-t) - \frac{\lambda_2}{\sqrt{xy}} \sqrt{s(y-t)(x-s)t}.
\end{aligned}
\end{equation*}
Similarly,
\begin{equation*}
\begin{aligned}
e(X-S, T) 
&\ge \frac{\sqrt{ab}}{\sqrt{|X| |Y|}} |X-S| |T| - \lambda_2 \sqrt{|X-S| |T| \left(1-\frac{|X-S|}{|X|}\right)\left(1-\frac{|T|}{|Y|}\right)}
\\&= \frac{\sqrt{ab}}{\sqrt{xy}}\cdot (x-s)t - \frac{\lambda_2}{\sqrt{xy}} \sqrt{(x-s)ts(y-t)}.
\end{aligned}
\end{equation*}
Thus
\begin{equation*}
\begin{aligned}
e(A_1, V-A_1) 
&= e(S, Y-T) + e(X-S, T)
\\&\ge \frac{\sqrt{ab}}{\sqrt{xy}}\cdot \left( s(y-t) +  (x-s)t \right) - \frac{2\lambda_2}{\sqrt{xy}} \sqrt{s(x-s)t(y-t)}.
\\&\ge \frac{\sqrt{ab}}{\sqrt{xy}}\cdot 2\sqrt{s(x-s)t(y-t)} - \frac{2\lambda_2}{\sqrt{xy}} \sqrt{s(x-s)t(y-t)}.
\\&= \frac{\sqrt{ab}-\lambda_2}{\sqrt{xy}}\cdot 2\sqrt{s(x-s)t(y-t)}.
\end{aligned}
\end{equation*}
Observe that $s(x-s)$ is minimized at $s=b-k_1$, and $t(y-t)$ is minimized at $t=a-k_2$. Thus $$e(A_1,V-A_1)\geq \frac{2(\sqrt{ab}-\lambda_2)}{\sqrt{xy}}\cdot \sqrt{(b-k_1)(x-(b-k_1))(a-k_2)(y-(a-k_2))}.$$
Since $k=k_1+k_2>0$, at least one of $k_1$ and $k_2$ is more than 0. Assume that $k_1>0$ (the case $k_2>0$ is similar). Then $x=|S|+|S'|+k_1> 2(b-k_1)$ and $y=|T|+|T'|+k_2\geq 2(a-k_2)$, and hence $1-\frac{b-k_1}{x}> \frac{1}{2}$ and $1-\frac{a-k_2}{y}\geq \frac{1}{2}$. 
Therefore,
\begin{equation}
\begin{aligned}\label{eqn:evcon}
e(A_1,V-A_1)
&\geq  \frac{2(\sqrt{ab}-\lambda_2)}{\sqrt{xy}}\cdot \sqrt{(b-k_1)(x-(b-k_1))(a-k_2)(y-(a-k_2))}\\
&= 2(\sqrt{ab}-\lambda_2)\cdot\sqrt{(b-k_1)\Big(1-\frac{b-k_1}{x}\Big)(a-k_2)\Big(1-\frac{a-k_2}{y}\Big)}\\
&> (\sqrt{ab}-\lambda_2)\cdot\sqrt{(a-k_2)(b-k_1)}.
\end{aligned}
\end{equation}
Notice that $k_1a +k_2b\le k_1\max\{a,b\} +k_2\max\{a,b\} \leq (k-1)\max\{a,b\}$, and so $(a-k_2)(b-k_1) = ab - (k_1a +k_2b) +k_1k_2\ge ab -(k-1)\max\{a,b\}$. By \eqref{eqn:kcon} and \eqref{eqn:evcon},
\begin{eqnarray*}
e(A_1,V-A_1)> (\sqrt{ab}-\lambda_2)\cdot\sqrt{ab -(k-1)\max\{a,b\}}\ge \frac{(k-1)\max\{a,b\}}{2}\ge \frac{k_1a +k_2b}{2},
\end{eqnarray*}
a contradiction. Thus $\kappa(G)\ge k$.
\end{proof}

To prove Theorem~\ref{thm:stp}, we need the spanning tree packing theorem. Let $(V_1, \ldots, V_t)$ be a sequence of disjoint vertex subsets of $V(G)$ and $e(V_1, \ldots, V_t)$ be the number of edges whose ends lie in different $V_i$'s. The following fundamental theorem on spanning tree packing number was established by Nash-Williams \cite{Nash-Williams} and Tutte \cite{Tutte}, independently.

\begin{theorem}[Nash-Williams~\cite{Nash-Williams} and Tutte~\cite{Tutte}]
\label{Nash-Williams-Tutte}
Let $G$ be a connected graph. Then $\tau(G)\geq k$ if and only if for any partition $(V_1, \ldots, V_t)$ of $V(G)$,
$$e(V_1, \ldots, V_t) \geq k(t-1).$$
\end{theorem}

\begin{proof}[\bf The proof of Theorem~\ref{thm:stp}]
Assume to the contrary that $\tau(G)\leq k-1$. By Theorem~\ref{Nash-Williams-Tutte}, there exists a partition $(V_1, V_2,\cdots, V_t)$ of $V(G)$  such that
\begin{equation}\label{equ::5}
\begin{aligned}
e(V_1, V_2,\cdots, V_t)\leq k(t-1)-1.
\end{aligned}
\end{equation}
Without loss of generality, we assume that $e(V_1, V-V_1)\leq e(V_2, V-V_2)\leq\cdots\leq e(V_t, V-V_t)$. Let $p$ be the largest index such that $e(V_p, V-V_p)\leq 2k-1$. We first claim that $p\ge 2$, for otherwise if $p\le 1$, then $e(V_2, V-V_2)\geq 2k$, and so 
$$e(V_1, V_2,\cdots, V_t)=\frac{1}{2}\sum_{1\leq i\leq t}e(V_i,V-V_i)\geq \frac{1}{2}\sum_{2\leq i\leq t}e(V_i,V-V_i)\geq k(t-1),$$
contradicting (\ref{equ::5}). Thus $p\geq 2$.

We assert that $V_i\cap X\neq \emptyset$ for $1\leq i\leq p$. If not, there exists some $j$ ($1\leq j\leq p$) such that $V_j\subseteq Y$, and thus, $e(V_j, V-V_j)= b|V_j| \ge 2k$, a contradiction. Similarly $V_i\cap Y\neq \emptyset$ for $1\leq i\leq p$. Let $V_i\cap X = S_i$, $V_i\cap Y = T_i$, $X-S_i =S_i'$ and $Y-T_i = T'_i$ where $1\leq i\leq p$. We next assert that $2|S_i|\ge b+1$ for $1\leq i\leq p$. Otherwise if $2|S_i|\le b$ for $1\leq i\leq p$, then 
\begin{equation*}
\begin{aligned}
e(V_i, V-V_i)&=a|S_i| + b|T_i|- 2e(S_i, T_i)\\
&\geq a|S_i| + b|T_i| - 2|S_i||T_i| \\
&\geq a|S_i| + b|T_i| -b|T_i| ~~~(\mbox{since $2|S_i|\le b$})\\
&=a|S_i|\\
&\geq 2k ~~(\mbox{since $a\geq 2k$ and $|S_i|\geq 1$}),
\end{aligned}
\end{equation*}
a contradiction. Similarly, we have $2|S'_i|\ge b+1$, $2|T_i|\ge a+1$ and $2|T'_i|\ge a+1$ for $1\leq i\leq p$.

Let $|X| =x$, $|Y| =y$, $|S_i| =s_i$ and $|T_i| =t_i$ where $1\leq i\leq p$. By Theorem~\ref{bieml}, for $1\leq i\leq p$, we have 
\begin{equation*}
\begin{aligned}
e(S_i, T'_i) 
&\ge \frac{\sqrt{ab}}{\sqrt{|X| |Y|}} |S_i| |T'_i| - \lambda_2 \sqrt{|S_i| |T'_i| \left(1-\frac{|S_i|}{|X|}\right)\left(1-\frac{|T'_i|}{|Y|}\right)}
\\&= \frac{\sqrt{ab}}{\sqrt{xy}}\cdot s_i(y-t_i) - \lambda_2 \sqrt{s_i(y-t_i) \left(1-\frac{s_i}{x}\right)\left(1-\frac{y-t_i}{y}\right)}
\\&= \frac{\sqrt{ab}}{\sqrt{xy}}\cdot s_i(y-t_i) - \frac{\lambda_2}{\sqrt{xy}} \sqrt{s_it_i(y-t_i)(x-s_i)}.
\end{aligned}
\end{equation*}
and
\begin{equation*}
\begin{aligned}
e(S'_i, T_i) 
&\ge \frac{\sqrt{ab}}{\sqrt{|X| |Y|}} |S'_i| |T_i| - \lambda_2 \sqrt{|S'_i| |T_i| \left(1-\frac{|S'_i|}{|X|}\right)\left(1-\frac{|T_i|}{|Y|}\right)}
\\&= \frac{\sqrt{ab}}{\sqrt{xy}}\cdot (x-s_i)t_i - \frac{\lambda_2}{\sqrt{xy}} \sqrt{s_it_i(x-s_i)(y-t_i)}.
\end{aligned}
\end{equation*}
Thus, for $1\leq i\leq p$,
\begin{equation}
\begin{aligned}\label{eq:evi}
e(V_i, V-V_i) 
 &= e(S_i, T'_i) + e(S'_i, T_i)\\
&\ge \frac{\sqrt{ab}}{\sqrt{xy}}\cdot \left(s_i(y-t_i) +  (x-s_i)t_i\right) - \frac{2\lambda_2}{\sqrt{xy}} \sqrt{s_i(x-s_i)t_i(y-t_i)}\\
&\ge \frac{\sqrt{ab}}{\sqrt{xy}}\cdot 2\sqrt{s_i(x-s_i)t_i(y-t_i)} - \frac{2\lambda_2}{\sqrt{xy}} \sqrt{s_i(x-s_i)t_i(y-t_i)}\\
&= \frac{2(\sqrt{ab}-\lambda_2)}{\sqrt{xy}}\cdot \sqrt{s_i(x-s_i)t_i(y-t_i)}.
\end{aligned}
\end{equation}
Note that $s_i(x-s_i)$ attains the minimum value at $s_i=\lceil\frac{b+1}{2}\rceil$ and $t_i(y-t_i)$ attains the minimum value at $t_i=\lceil\frac{a+1}{2}\rceil$ for $1\leq i\leq p$. Thus
\begin{equation*}
\begin{aligned}
e(V_1, V_2,\cdots, V_t) 
&=\frac{1}{2}\Big(\sum_{1\leq i\leq p}e(V_i, V-V_i) +\sum_{p+1\leq i\leq t}e(V_i, V-V_i)\Big)\\
&\ge\frac{p(\sqrt{ab}-\lambda_2)}{\sqrt{xy}}\cdot \sqrt{s_i(x-s_i)t_i(y-t_i)}+k(t-p)\\
&\ge \frac{p(\sqrt{ab}-\lambda_2)}{\sqrt{xy}}\cdot \sqrt{\Big\lceil\frac{b+1}{2}\Big\rceil\Big(x-\Big\lceil\frac{b+1}{2}\Big\rceil\Big)\Big\lceil\frac{a+1}{2}\Big\rceil\Big(y-\Big\lceil\frac{a+1}{2}\Big\rceil\Big)}+k(t-p).
\end{aligned}
\end{equation*}
Notice that $x\ge \sum_{1\le i\le p}|S_i| \ge p\lceil\frac{b+1}{2}\rceil$ and $y\ge \sum_{1\le i\le p}|T_i|\ge p\lceil\frac{a+1}{2}\rceil$. Then $1 -\frac{\lceil\frac{b+1}{2}\rceil}{x}\ge\frac{p-1}{p}$ and $1 -\frac{\lceil\frac{a+1}{2}\rceil}{y}\ge\frac{p-1}{p}$, and hence

\begin{equation*}
\begin{aligned}
e(V_1, V_2,\cdots, V_t) &\ge \frac{p(\sqrt{ab}-\lambda_2)}{\sqrt{xy}}\cdot \sqrt{\Big\lceil\frac{b+1}{2}\Big\rceil\Big(x-\Big\lceil\frac{b+1}{2}\Big\rceil\Big)\Big\lceil\frac{a+1}{2}\Big\rceil\Big(y-\Big\lceil\frac{a+1}{2}\Big\rceil\Big)}+k(t-p)\\
&= p(\sqrt{ab}-\lambda_2)\cdot \sqrt{\Big\lceil\frac{b+1}{2}\Big\rceil\Big(1-\frac{\lceil\frac{b+1}{2}\rceil}{x}\Big)\Big\lceil\frac{a+1}{2}\Big\rceil\Big(1-\frac{\lceil\frac{a+1}{2}\rceil}{y}\Big)}+k(t-p)\\
&\ge p(\sqrt{ab}-\lambda_2)\cdot \frac{p-1}{p}\cdot \sqrt{\Big\lceil\frac{b+1}{2}\Big\rceil\Big\lceil\frac{a+1}{2}\Big\rceil}+k(t-p)\\
&= (p-1)(\sqrt{ab}-\lambda_2)\cdot \sqrt{\Big\lceil\frac{b+1}{2}\Big\rceil\Big\lceil\frac{a+1}{2}\Big\rceil}+k(t-p)\\
&\ge (p-1)k + k(t-p) =k(t-1)
\end{aligned}
\end{equation*}
when $\lambda_2\le \sqrt{ab} -\frac{k}{\sqrt{\lceil\frac{a+1}{2}\rceil\lceil\frac{b+1}{2}\rceil}}$. This contradicts \eqref{equ::5}. Thus the result follows.
\end{proof}

\subsection{The proofs of Theorems~\ref{thm:rigid} and \ref{thm:globally-rigid}}
For any partition $\pi$ of $V(G)$, let $E_{G}(\pi)$  denote the set of edges in $G$ whose endpoints lie in different parts of $\pi$, and let $e_{G}(\pi)=|E_{G}(\pi)|$. A part is \textit{trivial} if it contains a single vertex. Let $Z\subset V(G)$, and let $\pi$ be a partition of $V(G-Z)$ with $n_{0}$ trivial parts $v_1,v_2,\ldots, v_{n_0}$. Denote by $n_{Z}(\pi)=\sum_{1\leq i\leq n_0}|Z_i|$, where $Z_i$ is the set of vertices in $Z$ that are adjacent to $v_i$ for $1\leq i\leq n_0$. For any $v\in V(G)$, let $N_{G}(v)$ and $d_G(v)$ be the neighborhood and degree of $v$ in $G$, respectively. For any vertex $v\in V(G)$ and any subset $S\subset V(G)$, let $d_{S}(v)=|N_{G}(v) \cap S|$.

\begin{lemma}[\cite{Gu-1}]\label{lem::3.1}
A graph $G$ contains $k$ edge-disjoint spanning rigid subgraphs if for every $Z\subset V(G)$ and every partition $\pi$ of $V(G-Z)$ with $n_{0}$ trivial parts and $n_{0}'$ nontrivial parts, $$e_{G-Z}(\pi)\geq k(3-|Z|)n_{0}'+2kn_0-3k-n_{Z}(\pi).$$
\end{lemma}


\begin{proof}[\bf The proof of Theorem~\ref{thm:rigid}]
Assume to the contrary that $G$ contains no $k$ edge-disjoint spanning rigid subgraphs. By Lemma \ref{lem::3.1}, there exist a subset $Z$ of $V$ and a partition $\pi$ of $V(G-Z)$ with $n_{0}$ trivial parts $v_1,v_2,\ldots, v_{n_0}$ and $n_{0}'$ nontrivial parts $V_1,V_2,\ldots,V_{n_0'}$ such that
\begin{equation}\label{equ::r1}
\begin{aligned}
e_{G-Z}(\pi)\leq k(3-|Z|)n_{0}'+2kn_0-3k-n_{Z}(\pi)-1,
\end{aligned}
\end{equation}
where $n_{Z}(\pi)=\sum_{1\leq i\leq n_0}|Z_i|$, and $Z_i$ is the set of vertices in $Z$ that are adjacent to $v_i$ for $1\leq i\leq n_0$. Let $V(G-Z)=V^{*}$ and let $\delta$ be the minimum degree of $G$. Then $\delta=\min\{a,b\}\geq 6k$. Observe that $d_{G-Z}(v_i)\geq \delta-|Z_i|\geq 6k-|Z_i|$ for $1\leq i\leq n_0$. Thus
\begin{equation*}
\begin{aligned}
e_{G-Z}(\pi)&= \frac{1}{2}\left(\sum_{1\leq i\leq n_{0}'}e(V_i, V^{*}-V_i)+\sum_{1\leq j\leq n_0}d_{G-Z}(v_j)\right)\\
&\geq \frac{1}{2}\left(\sum_{1\leq i\leq n_{0}'}e(V_i, V^{*}-V_i)+\delta n_0-\sum_{1\leq j\leq n_0}|Z_j|\right)\\
&\geq \frac{1}{2}\left(\sum_{1\leq i\leq n_{0}'}e(V_i, V^{*}-V_i)+6kn_0-n_{Z}(\pi)\right),
\end{aligned}
\end{equation*}
and hence
\begin{equation}\label{equ::r3}
\begin{aligned}
e_{G-Z}(\pi)&\geq 3kn_{0}-\frac{1}{2}n_{Z}(\pi).
           \end{aligned}
\end{equation}
We have the following two claims.

{\flushleft\bf Claim 1.} $|Z|\leq 2$.

Otherwise, $|Z|\geq 3$. By (\ref{equ::r1}),
\begin{equation*}
\begin{aligned}
e_{G-Z}(\pi)\leq k(3-|Z|)n_{0}'+2kn_0-3k-n_{Z}(\pi)-1\leq 2kn_0-3k-n_{Z}(\pi)-1.
\end{aligned}
\end{equation*}
Combining this with (\ref{equ::r3}) yields that $kn_0+3k+\frac{1}{2}n_{Z}(\pi)+1\leq 0$, which is impossible because $n_{0}\geq 0$, $k\geq 1$ and $n_{Z}(\pi)\geq 0$.

{\flushleft\bf Claim 2.} $n_{0}'\geq 2$.

Otherwise, $n_{0}'\leq 1$. By Claim 1 and (\ref{equ::r1}), we get
\begin{equation*}
\begin{aligned}
e_{G-Z}(\pi)\leq k(3-|Z|)n_{0}'+2kn_0-3k-n_{Z}(\pi)-1\leq 2kn_0-k|Z|-1-n_{Z}(\pi).
\end{aligned}
\end{equation*}
Combining this with (\ref{equ::r3}), we have $kn_0+1+\frac{1}{2}n_{Z}(\pi)+k|Z|\leq 0$, which is also impossible.

Assume that $e(V_1, V^{*}-V_1)\leq e(V_2, V^{*}-V_2)\leq \cdots\leq e(V_{n'_{0}}, V^{*}-V_{n'_{0}})$. Let $q$ be the largest index such that $e(V_q, V^{*}-V_q)\leq 6k-2k|Z|-1$. We first assert that $q\geq 2$. If not, $e(V_2, V^{*}-V_2)\geq 6k-2k|Z|$. Then
\begin{equation*}
\begin{aligned}
e_{G-Z}(\pi)&= \frac{1}{2}\left(\sum_{1\leq i\leq n_{0}'}e(V_i, V^{*}-V_i)+\sum_{1\leq j\leq n_0}d_{G-Z}(v_j)\right)\\
&\geq \frac{1}{2}\left((6k-2k|Z|)(n'_0-1)+6kn_0-n_{Z}(\pi)\right)\\
&= k(3-|Z|)n_{0}'+2kn_0-3k-n_{Z}(\pi)-1+kn_0+1+k|Z|+\frac{1}{2}n_{Z}(\pi)\\
&>k(3-|Z|)n_{0}'+2kn_0-3k-n_{Z}(\pi)-1~~(\mbox{since $|Z|\geq 0$, $n_{Z}(\pi)\geq 0$, $k\geq 1$ and $n_0\geq 0$}),
\end{aligned}
\end{equation*}
contradicting (\ref{equ::r1}). This implies that $q\geq 2$. We next assert that $V_i\cap X\neq \emptyset$ for $1\leq i\leq q$. If not, there exists some $j$ ($1\leq j\leq q$) such that $V_j\subseteq Y$. Combining this with Claim 1, $|V_j|\geq 2$, $k\geq 1$ and $b\geq 6k$, we have 
$$e(V_j, V^{*}-V_j)\geq e(V_j, V-V_j)-|Z||V_j|=(b-|Z|)|V_j| \ge 2(b-|Z|)> 6k-2k|Z|,$$ 
a contradiction. Similarly $V_i\cap Y\neq \emptyset$ for $1\leq i\leq q$. Let $V_i\cap X = S_i$, $V_i\cap Y = T_i$, $X-S_i =S_i'$ and $Y-T_i = T'_i$ where $1\leq i\leq q$. We also assert that $2|S_i|\ge b+1-|Z|$ for $1\leq i\leq q$. If not, there exists some $j$ ($1\leq j\leq q$) such that $2|S_j|\le b-|Z|$. Since $d_{V_j}(v)\leq \max\{|S_j|,|T_j|\}$ for each $v\in Z$, it follows that
\begin{equation*}
\begin{aligned}
e(V_j, V^{*}-V_j)&\geq a|S_j| + b|T_j|- 2e(S_j, T_j)-|Z|\max\{|S_j|,|T_j|\}\\
&\geq a|S_j| + b|T_j| - 2|S_j||T_j|-|Z|\max\{|S_j|,|T_j|\}\\
&\geq a|S_j| + |Z||T_j| -|Z|\max\{|S_j|,|T_j|\} ~~~(\mbox{since $2|S_j|\le b-|Z|$})\\
&\geq a~~(\mbox{since $0\leq |Z|\leq 2$, $|T_j|\geq 1$ and $|S_j|\geq 1$})\\
&\geq 6k-2k|Z|~~(\mbox{since $a\geq 6k$, $k\geq 1$ and $0\leq |Z|\leq 2$}),
\end{aligned}
\end{equation*}
a contradiction. Similarly, we have $2|S'_i|\ge b+1-|Z|$, $2|T_i|\ge a+1-|Z|$ and $2|T'_i|\ge a+1-|Z|$ for $1\leq i\leq q$. Let $|X| =x$, $|Y| =y$, $|S_i| =s_i$ and $|T_i| =t_i$ where $1\leq i\leq q$. By using the same argument as \eqref{eq:evi}, we have 
$$e(V_i, V-V_i)\geq \frac{2(\sqrt{ab}-\lambda_2)}{\sqrt{xy}}\cdot \sqrt{s_i(x-s_i)t_i(y-t_i)},$$
where $1\leq i\leq q$. Note that $s_i(x-s_i)$ is minimized at $s_i=\lceil\frac{b+1-|Z|}{2}\rceil$, and $t_i(y-t_i)$ is minimized at $t_i=\lceil\frac{a+1-|Z|}{2}\rceil$. Then for $1\leq i\leq q$,
$$e(V_i, V-V_i)\geq\frac{2(\sqrt{ab}\!-\!\lambda_2)}{\sqrt{xy}}\cdot \sqrt{\Big\lceil\frac{b+1\!-\!|Z|}{2}\Big\rceil\Big(x\!-\!\Big\lceil\frac{b+1\!-\!|Z|}{2}\Big\rceil\Big)\Big\lceil\frac{a\!+\!1-|Z|}{2}\Big\rceil\Big(y\!-\!\Big\lceil\frac{a+1\!-\!|Z|}{2}\Big\rceil\Big)}.$$
Notice that $x\ge \sum_{1\le i\le q}|S_i| \ge q\lceil\frac{b+1-|Z|}{2}\rceil$ and $y\ge \sum_{1\le i\le q}|T_i|\ge q\lceil\frac{a+1-|Z|}{2}\rceil$. Thus $1 -\frac{\lceil\frac{b+1-|Z|}{2}\rceil}{x}\ge\frac{q-1}{q}$ and $1 -\frac{\lceil\frac{a+1-|Z|}{2}\rceil}{y}\ge\frac{q-1}{q}$. Combining this with $d_{G}(v)\leq \max\{a,b\}$ for each $v\in Z$, $e(V_i, V-V_i)\geq e(V_i, V^*-V_i)\geq 6k-2k|Z|$ for $q+1\leq i\leq n'_0$  and $d_{G}(v_i)\geq 6k$ for $1\leq i\leq n_0$, we have
\begin{equation*}
\begin{aligned}
&e_{G-Z}(\pi)\\
&=\frac{1}{2}\left(\sum_{1\leq i\leq n_{0}'}e(V_i, V^{*}-V_i)+\sum_{1\leq j\leq n_0}d_{G-Z}(v_j)\right)\\
&\geq \frac{1}{2}\left(\sum_{1\leq i\leq n_{0}'}e(V_i, V-V_i)+\sum_{1\leq j\leq n_0}d_{G}(v_j)-|Z|\max\{a,b\}\right)\\
&\geq q(\sqrt{ab}-\lambda_2)\cdot\sqrt{\Big\lceil\frac{b+1-|Z|}{2}\Big\rceil\Big(1\!-\!\frac{\lceil\frac{b+1-|Z|}{2}\rceil}{x}\Big)\Big\lceil\frac{a+1-|Z|}{2}\Big\rceil\Big(1\!-\!\frac{\lceil\frac{a+1-|Z|}{2}\rceil}{y}\Big)}\!+\!k(n'_0\!-\!q)(3\!-\!|Z|)\\
&~~~+\frac{\sum_{1\leq j\leq n_0}d_{G}(v_j)}{2}-\frac{|Z|\max\{a,b\}}{2}\\
&\geq (q-1)(\sqrt{ab}-\lambda_2)\cdot\sqrt{\Big\lceil\frac{a+1-|Z|}{2}\Big\rceil\Big\lceil\frac{b+1-|Z|}{2}\Big\rceil}+k(n'_0-q)(3-|Z|)+3kn_0-\frac{|Z|\max\{a,b\}}{2}\\
&= k(3\!-\!|Z|)n_{0}'\!+\!2kn_0\!-\!3k\!-\!n_{Z}(\pi)-1\!+\!(q\!-\!1)(\sqrt{ab}\!-\!\lambda_2)\cdot\sqrt{\Big\lceil\frac{a\!+\!1\!-\!|Z|}{2}\Big\rceil\Big\lceil\frac{b\!+\!1\!-\!|Z|}{2}\Big\rceil}\!+\!kn_0+n_{Z}(\pi)\\
&~~+qk|Z|-\frac{|Z|\max\{a,b\}}{2}-3k(q-1)+1\\
&\geq k(3-|Z|)n_{0}'+2kn_0-3k-n_{Z}(\pi)-1+\Big(q-1-\frac{|Z|}{2}\Big)\max\{a,b\}+ qk|Z|+1\\
&~~(\mbox{since $\lambda_2\le \sqrt{ab} -\frac{3k+\max\{a,b\}}{\sqrt{\lceil\frac{a-1}{2}\rceil\lceil\frac{b-1}{2}\rceil}}$, $k\geq 1$, $n_0\geq 0$, $q\geq 2$ and $n_{Z}(\pi)\geq 0$})\\
&> k(3\!-\!|Z|)n_{0}'+2kn_0\!-\!3k\!-\!n_{Z}(\pi)-1~~(\mbox{since $q\geq 2$, $\max\{a,b\}\geq 6k$, $k\geq 1$ and $0\leq |Z|\leq 2$}),\\
\end{aligned}
\end{equation*}
contradicting (\ref{equ::r1}). Thus the result follows.
\end{proof}

  
\begin{proof}[\bf The proof of Theorem~\ref{thm:globally-rigid}]
By a combination of the results in \cite{Conn05, JaJo05}, a graph $G$ with more than $3$ vertices is globally rigid in $\mathbb{R}^2$ if and only if $G$ is 3-connected and redundantly rigid. Direct calculation yields that
\begin{eqnarray*}
\lambda_2\leq\sqrt{ab} -\frac{3+\max\{a,b\}}{\sqrt{\lceil\frac{a-2}{2}\rceil\lceil\frac{b-2}{2}\rceil}}< \sqrt{ab} -\frac{\max\{a,b\}}{\sqrt{ab - 2\max\{a,b\}}},
\end{eqnarray*}
due to $a\geq 6$ and $b\geq 6$. By Theorem~\ref{thm:conn}, $G$ is 3-connected.

It remains to show that $G$ is redundantly rigid. Suppose not, then there is an edge $f$ of $G$ such that $G-f$ is not rigid. Furthermore, by Lemma \ref{lem::3.1}, there exist a subset $Z$ of $V(G)$ and a partition $\pi$ of $V(G-f-Z)$ with $n_{0}$ trivial parts $v_1,v_2,\ldots, v_{n_0}$ and $n_{0}'$ nontrivial parts $V_1,V_2,\ldots,V_{n_0'}$ such that
\begin{equation}\label{equ::g1}
\begin{aligned}
e_{G-f-Z}(\pi)\leq (3-|Z|)n_{0}'+2n_0-4-n_{Z}(\pi).
\end{aligned}
\end{equation}
Note that $e_{G-f-Z}(\pi)\geq e_{G-Z}(\pi)-1$. Then
\begin{equation}\label{equ::g2}
\begin{aligned}
e_{G-Z}(\pi)\leq (3-|Z|)n_{0}'+2n_0-3-n_{Z}(\pi).
\end{aligned}
\end{equation}
Let $V(G-Z)=V^{*}$. Observe that $d_{G-Z}(v_i)\geq \min\{a,b\}-|Z_i|\geq 6-|Z_i|$ for $1\leq i\leq n_0$. Thus
\begin{equation}\label{equ::g3}
\begin{aligned}
e_{G-Z}(\pi)&= \frac{1}{2}\left(\sum_{1\leq i\leq n_{0}'}e(V_i, V^{*}-V_i)+\sum_{1\leq j\leq n_0}d_{G-Z}(v_j)\right)\geq 3n_{0}-\frac{1}{2}n_{Z}(\pi).
\end{aligned}
\end{equation}
We get the following two claims.

{\flushleft\bf Claim 1.} $|Z|\leq 2$.

Otherwise, $|Z|\geq 3$. Combining this with (\ref{equ::g2}) and (\ref{equ::g3}), we have
$n_0+3+\frac{1}{2}n_{Z}(\pi)\leq 0$. This is impossible because $n_{0}\geq 0$ and $n_{Z}(\pi)\geq 0$.

{\flushleft\bf Claim 2.} $n_{0}'\geq 2$.

Otherwise, $n_{0}'\leq 1$. By Claim 1, (\ref{equ::g2}) and (\ref{equ::g3}), we have
 $n_0+|Z|+\frac{1}{2}n_{Z}(\pi)\leq 0$. This implies that $n_0=0$, $n'_0=1$, $n_{Z}(\pi)=0$ and $|Z|=0$. From (\ref{equ::g1}), we obtain that 
$e_{G-f-Z}(\pi)\leq -1$, which is impossible.

Assume that $e(V_1, V^{*}-V_1)\leq e(V_2, V^{*}-V_2)\leq \cdots\leq e(V_{n'_{0}}, V^{*}-V_{n'_{0}})$. Let $q$ be the largest index such that $e(V_q, V^{*}-V_q)\leq 6-2|Z|$. We first assert that $q\geq 2$. If not, $e(V_2, V^{*}-V_2)\geq 7-2|Z|$. Then
\begin{equation*}
\begin{aligned}
e_{G-Z}(\pi)&= \frac{1}{2}\left(\sum_{1\leq i\leq n_{0}'}e(V_i, V^{*}-V_i)+\sum_{1\leq j\leq n_0}d_{G-Z}(v_j)\right)\\
&\geq \frac{1}{2}\left((7-2|Z|)(n'_0-1)+6n_0-n_{Z}(\pi)\right)\\
&\geq (3-|Z|)n_{0}'+2n_0-3-n_{Z}(\pi)+n_0+\frac{1}{2}(n'_0-1)+|Z|+\frac{1}{2}n_{Z}(\pi)\\
&>(3-|Z|)n_{0}'+2n_0-3-n_{Z}(\pi)~~(\mbox{since $|Z|\geq 0$, $n_{Z}(\pi)\geq 0$, $n'_0\geq 2$ and $n_0\geq 0$}),
\end{aligned}
\end{equation*}
contradicting (\ref{equ::g2}). It follows that $q\geq 2$. We next assert that $V_i\cap X\neq \emptyset$ for $1\leq i\leq q$. If not, there exists some $j$ ($1\leq j\leq q$) such that $V_j\subseteq Y$. Combining this with Claim 1, $|V_j|\geq 2$ and $b\geq 6$, we have 
$$e(V_j, V^{*}-V_j)\geq e(V_j, V-V_j)-|Z||V_j|=(b-|Z|)|V_j| \ge 2(b-|Z|)> 6-2|Z|,$$ 
a contradiction. Similarly $V_i\cap Y\neq \emptyset$ for $1\leq i\leq q$. Let $V_i\cap X = S_i$, $V_i\cap Y = T_i$, $X-S_i =S_i'$ and $Y-T_i = T'_i$ where $1\leq i\leq q$. We assert that $2|S_i|\ge b-|Z|$ for $1\leq i\leq q$. If not, there exists some $j$ ($1\leq j\leq q$) such that $2|S_j|\le b-|Z|-1$. Since $d_{V_j}(v)\leq \max\{|S_j|,|T_j|\}$ for each $v\in Z$, it follows that
\begin{equation*}
\begin{aligned}
e(V_j, V^{*}-V_j)&\geq a|S_j| + b|T_j|- 2e(S_j, T_j)-|Z|\max\{|S_j|,|T_j|\}\\
&\geq a|S_j| + b|T_j| - 2|S_j||T_j|-|Z|\max\{|S_j|,|T_j|\}\\
&\geq a|S_j| + (|Z|+1)|T_j| -|Z|\max\{|S_j|,|T_j|\} ~~~(\mbox{since $2|S_j|\le b-|Z|-1$})\\
&\geq a+1~~(\mbox{since $0\leq |Z|\leq 2$, $|T_j|\geq 1$ and $|S_j|\geq 1$})\\
&\geq 7-2|Z|~~(\mbox{since $a\geq 6$ and $0\leq |Z|\leq 2$}),
\end{aligned}
\end{equation*}
a contradiction. Similarly, we have $2|S'_i|\ge b-|Z|$, $2|T_i|\ge a-|Z|$ and $2|T'_i|\ge a-|Z|$ for $1\leq i\leq q$. Let $|X| =x$, $|Y| =y$, $|S_i| =s_i$ and $|T_i| =t_i$ where $1\leq i\leq q$. By using the same argument as \eqref{eq:evi}, we have 
$$e(V_i, V-V_i)\geq \frac{2(\sqrt{ab}-\lambda_2)}{\sqrt{xy}}\cdot \sqrt{s_i(x-s_i)t_i(y-t_i)},$$
where $1\leq i\leq q$. Note that $s_i(x-s_i)$ is minimized at $s_i=\lceil\frac{b-|Z|}{2}\rceil$, and $t_i(y-t_i)$ is minimized at $t_i=\lceil\frac{a-|Z|}{2}\rceil$. Then for $1\leq i\leq q$,
$$e(V_i, V-V_i)\geq\frac{2(\sqrt{ab}-\lambda_2)}{\sqrt{xy}}\cdot \sqrt{\Big\lceil\frac{b-|Z|}{2}\Big\rceil\Big(x\!-\!\Big\lceil\frac{b-|Z|}{2}\Big\rceil\Big)\Big\lceil\frac{a-|Z|}{2}\Big\rceil\Big(y\!-\!\Big\lceil\frac{a-|Z|}{2}\Big\rceil\Big)}.$$
Notice that $x\ge \sum_{1\le i\le q}|S_i| \ge q\lceil\frac{b-|Z|}{2}\rceil$ and $y\ge \sum_{1\le i\le q}|T_i|\ge q\lceil\frac{a-|Z|}{2}\rceil$. Thus $1 -\frac{\lceil\frac{b-|Z|}{2}\rceil}{x}\ge\frac{q-1}{q}$ and $1 -\frac{\lceil\frac{a-|Z|}{2}\rceil}{y}\ge\frac{q-1}{q}$. Combining this with $d_{G}(v)\leq \max\{a,b\}$ for each $v\in Z$, $e(V_i, V-V_i)\geq 7-2|Z|$ for $q+1\leq i\leq n'_0$  and $d_{G}(v_i)\geq 6$ for $1\leq i\leq n_0$, we have

\begin{equation*}
\begin{aligned}
&e_{G-Z}(\pi)\\
&\geq \frac{1}{2}\left(\sum_{1\leq i\leq n_{0}'}e(V_i, V-V_i)+\sum_{1\leq j\leq n_0}d_{G}(v_j)-|Z|\max\{a,b\}\right)\\
&\geq q(\sqrt{ab}-\lambda_2)\cdot\sqrt{\Big\lceil\frac{b-|Z|}{2}\Big\rceil\Big(1\!-\!\frac{\lceil\frac{b-|Z|}{2}\rceil}{x}\Big)\Big\lceil\frac{a-|Z|}{2}\Big\rceil\Big(1\!-\!\frac{\lceil\frac{a-|Z|}{2}\rceil}{y}\Big)}\!+\!(n'_0\!-\!q)\Big(\frac{7}{2}\!-\!|Z|\Big)+3n_0\\
&~~~-\frac{|Z|\max\{a,b\}}{2}\\
&\geq (q-1)(\sqrt{ab}-\lambda_2)\cdot\sqrt{\Big\lceil\frac{a-|Z|}{2}\Big\rceil\Big\lceil\frac{b-|Z|}{2}\Big\rceil}+(n'_0\!-\!q)\Big(\frac{7}{2}\!-\!|Z|\Big)+3n_0-\frac{|Z|\max\{a,b\}}{2}\\
&= (3\!-\!|Z|)n_{0}'\!+\!2n_0\!-\!3\!-\!n_{Z}(\pi)\!+\!(q\!-\!1)(\sqrt{ab}\!-\!\lambda_2)\cdot\sqrt{\Big\lceil\frac{a\!-\!|Z|}{2}\Big\rceil\Big\lceil\frac{b\!-\!|Z|}{2}\Big\rceil}\!+\!n_0+n_{Z}(\pi)+q|Z|\\
&~~-\frac{|Z|\max\{a,b\}}{2}-3(q-1)+\frac{n_{0}'-q}{2}\\
&\geq (3-|Z|)n_{0}'+2n_0-3-n_{Z}(\pi)+\Big(q-1-\frac{|Z|}{2}\Big)\max\{a,b\}+ q|Z|+n_0+\frac{n_{0}'-q}{2}\\
&~~(\mbox{since $\lambda_2\le \sqrt{ab} -\frac{3+\max\{a,b\}}{\sqrt{\lceil\frac{a-2}{2}\rceil\lceil\frac{b-2}{2}\rceil}}$, $n_0\geq 0$, $q\geq 2$ and $n_{Z}(\pi)\geq 0$})\\
&> (3\!-\!|Z|)n_{0}'\!+\!2n_0\!-\!3\!-\!n_{Z}(\pi)~~(\mbox{since $q\geq 2$, $\max\{a,b\}\geq 6$, $n_0\geq 0$, $n_{0}'\geq q$ and $0\leq |Z|\leq 2$}),\\
\end{aligned}
\end{equation*}
contradicting (\ref{equ::g2}). Thus the result follows.
\end{proof}

\section{Final remarks}
By the bipartite expander mixing lemma, biregular bipartite graphs are kinds of pseudorandom graphs and expanders. In this paper, we proved several spectral expansion properties of biregular bipartite graphs via the second largest adjacency eigenvalue. There are two related problems that we would like to mention.

Defined by Chv\'atal \cite{Chva73} in 1973, the {\em toughness} $t(G)$ of a connected non-complete graph $G$ is defined as $t(G)=\min\left\{\frac{|S|}{\omega(G-S)}\right\}$, in which the minimum is taken over all proper subsets $S$ of $V(G)$ such that $G-S$ is disconnected and $\omega(G-S)$ denotes the number of components of $G-S$. Toughness of regular graphs from eigenvalues was first studied by Alon~\cite{Alon95} who proved that for any connected $d$-regular graph $G$, $t(G)>\frac{1}{3}\left(\frac{d^2}{d\lambda+\lambda^2} -1\right)$. Brouwer~\cite{Brou95} independently showed that $t(G)>\frac{d}{\lambda}-2$ for any connected $d$-regular graph $G$, and he also conjectured that $t(G) \ge\frac{d}{\lambda}-1$ in \cite{Brou95, Brou96}. The conjecture has been settled in~\cite{Gu21b} and was extended to general graphs in~\cite{GH22}. For bipartite graphs, obviously the toughness is at most $1$, and some spectral conditions were obtained in \cite{CiWo14,CiGu16}. For $(a,b)$-biregular bipartite graphs, the upper bound on toughness can be improved to $\min\{a/b,b/a\}$. It would be interesting to investigate toughness of biregular bipartite graphs via eigenvalues.

Let $c$ be a positive integer. A {\it $c$-forest} is a forest with exactly $c$ components. The {\it $c$th-order edge toughness} of a graph $G$ is defined as $\tau_c(G) = \min\left\{\frac{|F|}{\omega(G-F)-c} \right\}$, where $\omega(G-F)$ denotes the number of components of $G-F$ and the minimum is taken over all subsets $F$ of $E(G)$ such that $\omega(G-F)>c$. Generalizing the spanning tree packing theorem, Chen, Koh and Peng~\cite{ChKP93} proved that for positive integers $k$ and $c =1, 2, \cdots, |V(G)|-1$, a graph $G$ has $k$ edge-disjoint spanning $c$-forests if and only if $\tau_c(G) \ge k$. For $c=1$, this is a reformulation of Theorem~\ref{Nash-Williams-Tutte}. Cioab\u{a} and Wong~\cite{CiWo12} proposed an open problem to find connections between eigenvalues and $\tau_c(G)$ for $c\ge 2$. It is interesting to study this question for biregular bipartite graphs.

\subsection*{Acknowledgements}
Dandan Fan was supported by the National Natural Science Foundation of China (No.\,12301454) and the Natural Science Foundation of Xinjiang Uygur Autonomous Region (No.\,2022D01B103), Xiaofeng Gu was supported by a grant from the Simons Foundation (No.\,522728), and Huiqiu Lin was supported by the National Natural Science Foundation of China (Nos.\,12271162 and \,12326372), Natural Science Foundation of Shanghai (Nos.\,22ZR1416300 and \,23JC1401500) and the Program for Professor of Special Appointment (Eastern Scholar) at Shanghai Institutions of Higher Learning (No.\,TP2022031).


\begin{thebibliography}{99}

\bibitem{Alon86} 
N. Alon, Eigenvalues and Expanders, \emph{Combinatorica} \textbf{6} (1986) 83--96.

\bibitem{Alon95}
N. Alon, Tough Ramsey graphs without short cycles, \emph{J. Algebraic Combin.} \textbf{4} (1995) 189--195.

\bibitem{AlCh88}
N. Alon and F.R.K. Chung, Explicit construction of linear sized tolerant networks, \emph{Discrete Math.} \textbf{72} (1988) 15--19.




\bibitem{BDH22}
G. Brito, I. Dumitriu and K.D. Harris, Spectral gap in random bipartite biregular graphs and applications, \emph{Combin. Probab. Comput.} \textbf{31(2)} (2022) 229--267.

\bibitem{Brou95}
A.E. Brouwer, Toughness and spectrum of a graph, \emph{Linear Algebra Appl.} \textbf{226/228} (1995) 267--271.

\bibitem{Brou96}
A.E. Brouwer, Spectrum and connectivity of graphs, \emph{CWI Quarterly} \textbf{9} (1996) 37--40.



\bibitem{BV20}
S.P. Burnwal and M. Vidyasagar, Deterministic completion of rectangular matrices using asymmetric Ramanujan graphs: Exact and stable recovery, \emph{IEEE Transactions on Signal Processing} \textbf{68} (2020) 3834--3848.

\bibitem{ChKP93}
C.C. Chen, K.M. Koh and Y.H. Peng, On the higher-order edge toughness of a graph, \emph{Discrete Math.} \textbf{111} (1993) 113--123.

\bibitem{Chva73}
V. Chv\'atal, Tough graphs and hamiltonian circuits, \emph{Discrete Math.} \textbf{5} (1973) 215--228.

\bibitem{Cioa10}
S.M. Cioab\u{a}, Eigenvalues and edge-connectivity of regular graphs, \emph{Linear Algebra Appl.} \textbf{432} (2010) 458--470.

\bibitem{CDGG23}
S.M.~Cioab\u{a}, S. Dewar, G. Grasegger and X. Gu, Graph rigidity properties of Ramanujan graphs, \emph{Electron. J. Combin.} \textbf{30(3)} (2023) P3.12.

\bibitem{CDG21}
S.M.~Cioab\u{a}, S. Dewar and X. Gu, Spectral conditions for graph rigidity in the Euclidean plane,  {\em Discrete Math.} {\bf 344} (2021), 112527.

\bibitem{CiGu16} 
S.M. Cioab\u{a} and X. Gu, Connectivity, toughness, spanning trees of bounded degrees, and spectrum of regular graphs, \emph{Czechoslovak Math. J.} \textbf{66} (2016) 913--924.

\bibitem{coppww22}
S.M. Cioab\u{a}, A. Ostuni, D. Park, S. Potluri, T. Wakhare and W. Wong, Extremal graphs for a spectral inequality on edge-disjoint spanning trees, \emph{Electron. J. Combin.} \textbf{29} (2022) P2.56.

\bibitem{CiWo12}
S.M. Cioab\u{a} and W. Wong, Edge-disjoint spanning trees and eigenvalues of regular graphs, \emph{Linear Algebra Appl.} \textbf{437} (2012) 630--647.

\bibitem{CiWo14} S.M. Cioab\u{a} and W. Wong, The spectrum and toughness of regular graphs, \emph{Discrete Appl. Math.} \textbf{176} (2014) 43--52.

\bibitem{Conn05}
R. Connelly, Generic global rigidity, \emph{Discrete Comput. Geom.} \textbf{33} (4) (2005) 549--563.

\bibitem{DeSV12}
S. De Winter, J. Schillewaert and J. Verstraete, Large incidence-free sets in geometries, \emph{Electron. J. Combin.} \textbf{19(4)} (2012) \#P24.

\bibitem{DJ16}
I. Dumitriu and T. Johnson, The Mar\v{c}enko-Pastur law for sparse random bipartite biregular graphs, \emph{Random Structures Algorithms} \textbf{48(2)} (2016) 313--340.


\bibitem{EFMP23}
S. Evra, B. Feigon, K. Maurischat and O. Parzanchevski, Ramanujan Bigraphs, arXiv preprint arXiv:2312.06507 (2023).

\bibitem{Fan-Gu-Lin}
D. Fan, X. Gu and H. Lin, Spectral radius and edge‐disjoint spanning trees, \textit{J. Graph Theory} \textbf{104} (2023) 697--711.

\bibitem{FL96}
K. Feng and W.-C. W. Li, Spectra of hypergraphs and applications, \emph{J. Number Theory} \textbf{60(1)} (1996) 1--22.

\bibitem {F73}
M. Fiedler,  Algebraic connectivity of graphs, {\em Czechoslovak Math. J.} {\bf 23} (1973) 298--305. 

\bibitem{GM88}
C.D. Godsil and B. Mohar, Walk generating functions and spectral measures of infinite graphs, \emph{Linear Algebra Appl.} \textbf{107} (1988) 191--206.


\bibitem{GortlerThurstonHealey10}
S.J.~Gortler, A.~Healy, D.~Thurston, Characterizing generic global rigidity, \emph{Amer. J. Math.} {\bf 132(4)} (2010) 897--939.


\bibitem{Gu-1}
X. Gu, Spanning rigid subgraph packing and sparse subgraph covering, \emph{SIAM J. Discrete Math.} \textbf{32(2)} (2018) 1305--1313.

\bibitem{Gu21}
X. Gu, Toughness in pseudo-random graphs, \emph{European J. Combin.} \textbf{92} (2021) 103255.

\bibitem{Gu21b}
X. Gu, A proof of Brouwer's toughness conjecture, \emph{SIAM J. Discrete Math.} \textbf{35} (2021) 948--952.

\bibitem{GH22}
X. Gu and W.H. Haemers, Graph toughness from Laplacian eigenvalues, \emph{Algebraic Combinatorics} \textbf{5(1)} (2022) 53--61.


\bibitem{GLLY12}
X. Gu, H.-J. Lai, P. Li and S. Yao, Edge-disjoint spanning trees, edge connectivity and eigenvalues in graphs, \emph{J. Graph Theory} \textbf{81} (2016) 16--29.

\bibitem{GMRWY21}
X. Gu, W. Meng, M. Rolek, Y. Wang and G. Yu, Sufficient conditions for 2-dimensional global rigidity, \emph{SIAM J. Discrete Math.} \textbf{35(4)} (2021) 2520--2534.
 
\bibitem{Haem79}
W.H. Haemers, Eigenvalue techniques in design and graph theory, PhD thesis, 1979.

\bibitem{Haem95}
W.H. Haemers, Interlacing eigenvalues and graphs, \emph{Linear Algebra Appl.} \textbf{226-228} (1995) 593--616.


\bibitem{HLW06}
S. Hoory, N. Linial and A. Wigderson, Expander graphs and their applications, \emph{Bulletin of the American Mathematical Society} \textbf{43} (2006) 439--562.

\bibitem{JaJo05} 
B.~Jackson and T.~Jord\'an, Connected rigidity matroids and unique realizations of graphs, {\em J. Combin. Theory Ser. B} {\bf 94} (2005) 1--29.

\bibitem{JaJo09} 
B.~Jackson and T.~Jord\'an, A sufficient connectivity condition for generic rigidity in the plane, {\em Discrete Appl. Math.} {\bf 157} (2009) 1965--1968.

\bibitem{KrSu06}
M. Krivelevich and B. Sudakov, Pseudo-random graphs, More sets, graphs and numbers, \emph{Bolyai Soc. Math. Stud.} 15, Springer, Berlin, 2006, 199--262.

\bibitem{Lama70} G.~Laman, On graphs and rigidity of plane skeletal structures, {\em J. Engrg. Math.} {\bf 4} (1970) 331--340.

\bibitem{LS96}
W.-C.W. Li and Patrick Sol\'e, Spectra of regular graphs and hypergraphs and orthogonal polynomials, \emph{European J. Combin.} \textbf{17(5)} (1996) 461--477.

\bibitem{LiHL13} Q. Liu, Y. Hong and H.-J. Lai, Edge-disjoint spanning trees and eigenvalues, \emph{Linear Algebra Appl.} \textbf{444} (2014) 146--151.

\bibitem{LHGL14}
Q. Liu, Y. Hong, X. Gu and H.-J. Lai, Note on edge-disjoint spanning trees and eigevalues, \emph{Linear Algebra Appl.} \textbf{458} (2014) 128--133.

\bibitem{LoYe82}
L.~Lov\'asz and Y.~Yemini, On generic rigidity in the plane, {\em SIAM J. Algebraic Discrete Methods} {\bf 3} (1982) 91--98.

\bibitem{MSS15}
A.W. Marcus, D.A. Spielman and N. Srivastava, Interlacing families I: bipartite Ramanujan graphs of all degrees, \emph{Ann. of Math.} \textbf{182} (2015) 307--325.

\bibitem{Nash-Williams}
C.St.J.A. Nash-Williams, Edge-disjoint spanning trees of finite graphs, \emph{J. Lond. Math. Soc.} \textbf{36} (1961) 445--450.

\bibitem{Poll1927}
H.~Pollaczek-Geiringer, \"{U}ber die Gliederung ebener Fachwerke, {\em Z. Angew. Math. Mech.} {\bf 7(1)} (1927) 58--72.

\bibitem{SS96}
M. Sipser and D.A. Spielman, Expander codes, \emph{IEEE transactions on Information Theory} \textbf{42} (1996) 1710--1722.

\bibitem{Thom85}
A. Thomason, Pseudo-random graphs, in: Proceedings of Random Graphs, Pozna\'n 1985, M. Karo\'nski, ed., \emph{Ann. Discrete Math.} \textbf{33} (1987) 307--331.

\bibitem{Thom87}
A. Thomason, Random graphs, strongly regular graphs and pseudo-random graphs, Surveys in Combinatorics, C. Whitehead, ed., LMS Lecture Note Series \textbf{123} (1987) 173--195.

\bibitem{Tran20}
L.V. Tran, Local Law for Eigenvalues of Random Regular Bipartite Graphs, \emph{Bull. Malays. Math. Sci. Soc.} \textbf{43(2)} (2020) 1517--1526.

\bibitem{Tutte}
W.T. Tutte, On the problem of decomposing a graph into $n$ factors, \emph{J. Lond. Math. Soc.} \textbf{36} (1961) 221--230.


\bibitem{Zhu23}
Y. Zhu, On the second eigenvalue of random bipartite biregular graphs, \emph{J. Theoret. Probab.} \textbf{36} (2023) 1269--1303.


\end{thebibliography}
\end{document}